\documentclass[11pt]{article}
\usepackage[T1]{fontenc}
\usepackage[utf8]{inputenc}
\usepackage[russian]{babel}
\usepackage{amsfonts,amsbsy,amssymb,amsmath,amsthm}
\usepackage{misccorr}

\makeatletter
\def\section{\@startsection{section}{1}{\z@}{3.5ex plus 1ex minus .2ex}
    {2.3ex plus .2ex}{\center\normalfont\large\bfseries\S\,}}
\def\subsection{\@startsection{subsection}{2}{\z@}%
    {1.5ex plus 1ex minus .2ex}{-1.5ex plus
        -.2ex}{\indent\normalfont\bfseries}}%
\def\subsubsection{\@startsection{subsubsection}{3}{\z@}%
    {-1ex plus -1ex minus -.2ex}{-1.5ex plus -.2ex}%
    {\normalsize\indent\normalfont\bfseries}}
\makeatother

\newtheorem{thm}{Теорема}
\newtheorem{thmeng}{Theorem}
\newtheorem{lm}{Лемма}
\newtheorem{conj}{Гипотеза}
\newtheorem{conjeng}{Conjecture}
\newtheorem{cons}{Следствие}

\newtheorem{st}{Утверждение}
\theoremstyle{definition}
\newtheorem{definition}{Определение}

\begin{document}
    
    \renewcommand{\refname}{Литература}
\begin{center}
    {\large\bf Maximal Lie subalgebras of locally nilpotent derivations}
\end{center}
\begin{center}
{\large\bf A. Skutin}
\end{center}
    \section{Introduction}
Let $B$ denote a commutative $\mathbb{K}$-domain, where $\mathbb{K}$ is any field of characteristic zero. In the present article we will be interested in the Lie subalgebras $\mathfrak{g}$ of $\text{LND}(B)$, by which we mean subsets $\mathfrak{g}$ of locally nilpotent derivations of algebra $B$ which are Lie subalgebras of the set of all derivations of $B$.
G. Freudenburg \cite[11.7, p. 238]{Fr} posed the next conjectures about the structure of maximal Lie subalgebras of $\text{LND}(B)$.
\begin{conjeng}
    The triangular subalgebra 
$\frak{T} = \mathbb{K}\partial x_1\oplus \ldots\oplus \mathbb{K}[x_1, \ldots, x_{n-1}]\partial x_n$ is the maxiamal Lie algebra of $\text{LND}(\mathbb{K}[x_1, \ldots, x_n])$.
\end{conjeng}
\begin{conjeng}
    
    $\frak{T}$ is the unique maximal subalgebra of $\text{LND}(\mathbb{K}[x_1, \ldots, x_n])$ up to conjugation.
    
\end{conjeng}
In this article, we prove that conjecture 1 is true, and conjecture 2 is false for $n> 2$. However, under some natural ussumptions about the Lie subalgebra $\mathcal{A}$ of $\text{LND}(\mathbb{K}[x_1, \ldots, x_n])$, conjecture 2 can be proved. In fact, we prove the following theorem:
\begin{thmeng}
    
    A maximal Lie subalgebra $\mathcal{A}$ of $\text{LND}(\mathbb{K}[x_1,\ldots , x_n])$  is conjugated to $\mathbb{K}\partial x_1\oplus \ldots \oplus \mathbb{K}[x_1, \ldots , x_{n-1}]\partial x_n$ if and only if $\ker\mathcal{A} := \cap_{D\in\mathcal{A}}\ker D=\mathbb{K}$.
    
\end{thmeng}
Author is greateful to I.V. Arzhantsev for stating the problem and attention to work.
\newpage
\begin{center}
    {\large\bf Максимальные подалгебры Ли среди локально нильпотентных дифференцирований}
\end{center}
\begin{center}
{\large\bf А. Скутин}
\end{center}
    \setcounter{section}{0}
    \section{Введение}
    
    Рассмотрим произвольную коммутативную алгебру $B$ без делителей нуля, имеющую конечную степень трансцендентности над полем нулевой характеристики $\mathbb{K}$. Далее в работе всегда будем рассматривать только такие алгебры. Нас будут интересовать алгебры Ли содержащиеся среди локально нильпотентных дифференцирований алгебры $B$.
    
    В своей монографии \cite{Fr} Джином Фройденбургом были высказаны следующие гипотезы о строении максимальных по вложению подалгебр Ли среди локалньно нильпотентных дифференцирований алгебры $B$ в случае, когда $B$~-- алгебра многочленов от $n$ переменных \cite[11.7; стр. 238]{Fr}.
    
    \begin{conj}
        
        Алгебра Ли $\frak{T} = \mathbb{K}\partial x_1\oplus \ldots\oplus \mathbb{K}[x_1, \ldots, x_{n-1}]\partial x_n$ является максимальной по вложению алгеброй Ли, лежащей в $\text{LND}(\mathbb{K}[x_1, \ldots, x_n])$.
        
    \end{conj}
    
    \begin{conj}
        
        Все максимальные алгебры Ли $\mathcal{A}\subset \text{LND}(\mathbb{K}[x_1,\ldots , x_n])$ сопряжены алгебре Ли $\mathbb{K}\partial x_1\oplus \ldots \oplus \mathbb{K}[x_1, \ldots , x_{n-1}]\partial x_n$.
        
    \end{conj}
        
    Как мы покажем дальше, гипотеза 1 верна, а гипотеза 2 не верна в общем случае. Однако, на алгебру Ли $\mathcal{A}$ можно наложить некоторые естественные дополнительные условия, при которых гипотеза 2 всё же имеет место. В этой работе мы докажем, что гипотеза 2 верна в следующей более слабой форме:
    
    \begin{thm}
        
        Пусть дана максимальная алгебра Ли $\mathcal{A}\subset\text{LND}(\mathbb{K}[x_1,\ldots , x_n])$ и известно, что $\ker\mathcal{A} := \cap_{D\in\mathcal{A}}\ker D=\mathbb{K}$. Тогда $\mathcal{A}$ сопряжена алгебре Ли $\mathbb{K}\partial x_1\oplus \ldots \oplus \mathbb{K}[x_1, \ldots , x_{n-1}]\partial x_n$.
        
    \end{thm}
    
    Автор выражает благодарность И.В. Аржанцеву за постановку задачи и внимание к работе.

    \section{Используемые утверждения}
    
    Будем использовать следующие обозначения:
    
    ~--- Для каждой алгебры Ли $\mathcal{A}$ определим $d(\mathcal{A})$ как её ступень разрешимости. Также, обозначим за $\mathcal{A}^{(i)}$~-- $i$-ый коммутант алгебры Ли $\mathcal{A}$;
    
    ~--- Для произвольного семейства линейных операторов $S$ векторного пространства $V$, ядром $\ker S$ этого семейства будем называть пересечения ядер всех операторов, лежащих в множестве $S$;
    
    ~--- Для произвольных двух семейств линейных операторов $S_1$, $S_2$ векторного пространства $V$ обозначим за $[S_1, S_2]$~-- множество всевозможных коммутаторов вида $[A, B] = AB - BA$, где $A\in S_1$, $B\in S_2$;
    
    ~--- Для произвольного семейства линейных операторов $S$ векторного пространства $V$ и произвольного элемента $v\in V$ обозначим за $S(v)$ множество элементов вида $A(v)$, где $A$--всевозможные операторы из $S$;
    
    ~--- Для произвольного семейства линейных операторов $S$ векторного пространства $V$ обозначим за $\text{ad}S$ множество присоединённых эндоморфизмов $\{\text{ad}A: X\to [A, X]| A\in S, X\in\text{End}(V)\}$ векторного пространства $\text{End}(V)$.
    
    \begin{definition}
        Линейный оператор $A$ на векторном пространстве $V$ называется \textit{локально нильпотентным}, если для каждого вектора $v$ из $V$ найдётся натуральное число $k = k(v)>0$ такое, что $A^k(v) = 0$.
    \end{definition}
    
    \begin{definition}
        
        Множество линейных операторов $T$ на векторном пространстве $V$ будем называть \textit{локально нильпотентным}, если для любой бесконечной последовательности операторов $A_1, A_2,\ldots$, лежащих в $T$, и любого вектора $v\in V$, существует такое $k=k(v, \{A_i\})>0$, что выполнено равенство $A_k\ldots A_1(v) = 0$.
        
    \end{definition}
    
    \begin{lm}
        Пусть $T$~-- локально нильпотентное множество линейных операторов на векторном пространстве $V$. Тогда для произвольного собственного векторного подпространства $U\subset V$ найдётся элемент $v\in V\setminus U$, для которого $T(v)\subseteq U$.
        
    \end{lm}
    
    \begin{proof}
        Рассмотрим произвольный вектор $v$ лежащий в множестве $V\setminus U$. Из того, что множество $T$ является локально нильпотентным имеем, что для некоторого конечного числа элементов $A_1$, $A_2$, $\ldots$, $A_k$, лежащих в множестве $T$, элемент $A_k\ldots A_1(v)$ не принадлежит множеству $U$, однако, элементы $TA_k\ldots A_1(x)$ лежат в множестве $U$. Теперь, в качестве искомого элемента достаточно рассмотреть элемент $v'=A_k\ldots A_1(v)$.
    \end{proof}
    
    \begin{lm}
        
        Рассмотрим произвольное локально нильпотентное множество линейных операторов $T$ на векторном пространстве $V$. Тогда для любого локально нильпотентного оператора $A$ на векторном пространстве $V$ такого, что $[A, T]\subseteq T$, имеем $T\cup\{A\}$~-- локально нильпотентное множество линейных операторов на $V$.
        
    \end{lm}
    
    \begin{proof}
        
    Пусть $U$~-- максимальное по вложению линейное подпространство $V$ инвариантное для каждого линейного оператора из множества $T\cup\{A\}$ такое, что ограничение $T\cup\{A\}$ на $U$ является локально нильпотентным множеством линейных операторов. Предположим, что $U\not= V$. По лемме 1 существует элемент $v\in V\setminus U$ такой, что $Tv\subseteq U$. Несложно проверить, что линейное пространство $U'=\left<T, v, Av, A^2v,\ldots \right>$ является инвариантным для каждого линейного поератора из $T\cup\{A\}$ и ограничение $T\cup\{A\}$ на $U'$ является локально нильпотентным множеством линейных операторов. Противоречие к максимальности $U$.
    \end{proof}
    
    Каждое локально нильпотентное дифференцирование алгебры $B$ является также локально нильпотентным линейным оператором на векторном пространстве $B_{\mathbb{K}}$. Поэтому, все предыдущие леммы о локально нильпотентных операторах переносятся на локально нильпотентные дифференцирования, причем, в случае с дифференцированиями алгебры $B$, локально нильпотентным множеством дифференцирований алгебры $B$ называется произвольное множество дифференцирований алгебры $B$ являющееся локально нильпотентным множеством линейных операторов на векторном пространстве $B_{\mathbb{K}}$.
    
    \begin{lm}
        
        Множество $\mathbb{K} \partial x_1\oplus \ldots \oplus \mathbb{K}[x_1, \ldots , x_{n-1}]\partial x_n$ является локально нильпотентным множеством дифференцирований алгебры $\mathbb{K}[x_1,\ldots, x_n]$.
        
    \end{lm}
    
    \begin{proof}
        
        Упорядочим переменные $x_1\prec\ldots\prec x_n$ и введём лексикографический порядок на мономах. При действии дифференцирований из $\mathbb{K}\partial x_1\oplus \ldots \oplus \mathbb{K}[x_1, \ldots , x_{n-1}]\partial x_n$ на произвольный многочлен, порядок старшего монома этого многочлена уменьшается, откуда получаем доказательство леммы. \end{proof}
    
    \begin{lm}
        
        Рассмотрим произвольное локально нильпотентное множество дифференцирований $S$ алгебры $B$. Тогда для любого конечного множества локально нильпотентных дифференцирований $D_1, D_2,\ldots, D_k$ такого, что $[S\cup\{D_1, D_2, \ldots, D_k\}, S\cup\{D_1, D_2, \ldots, D_k\}]\subseteq S$, имеем $S\cup\{D_1, D_2, \ldots, D_k\}$~-- локально нильпотентное подмножество дифференцирований алгебры $B$.
        
    \end{lm}
    
    \begin{proof}
        Для доказательсва леммы 4 достаточно воспользоваться индукцией по $k$ и леммой 2.
    \end{proof}
    
    \begin{lm}
        
        Рассмотрим произвольное подмножество $S$, лежащее в $\text{Der}(B)$. Тогда для некоторого конечного числа элементов $D_1, \ldots , D_k\in S$, $\ker S = \cap_{i=1}^k\ker D_i$. Более того, $k$ можно выбрать равным $\text{tr.deg.}(B)-\text{tr.deg.}(\ker S)$.
        
    \end{lm}
    
    \begin{proof}
        
        Предположим противное. Тогда найдётся счётное семейство элементов $\{D_i\in S\}_{i=1}^{\infty}$ такое, что $\cap_{i=1}^{l+1}\ker D_i\subsetneqq\cap_{j=1}^l\ker D_j$, для всех $l$. Построим бесконечную цепочку элементов $\{x_l\in (\cap_{i=1}^l\ker D_i)\setminus (\cap_{j=1}^{l+1}\ker D_j)\}_{l=1}^{\infty}$. Применяя \cite[Prop. 1.8(d); стр. 17]{Fr}, получаем, что $x_1, x_2,\ldots$ образует бесконечное алгебраически независимое семейство элементов, что противоречит условию того, что $\text{tr.deg.}(B)<\infty$.
    \end{proof}
    
    \begin{lm}
        
        Любая коммутативная алгебра Ли $\mathcal{A}$, лежащая в $\text{LND}(B)$, является конечномерной алгеброй Ли размерности не более чем $\text{tr.deg.}_{\mathbb{K}}(B)$.
        
    \end{lm}
    
    \begin{proof}
        
        Это утверждение является следствием известного факта о том, что не существует более чем $\text{tr.deg.}B$ линейно независимых локально нильпотентных дифференцирований алгебры $B$.
    \end{proof}
    
    \section{Опровержение гипотезы 2 в общем случае}
    
    \begin{thm}
        Пусть $D$~-- локально нильпотентное дифференцирование алгебры многочленов \\$\mathbb{K}[x_1,\ldots , x_n]$ такое, что абелева алгебра Ли $(\ker D)D$ является подалгеброй алгебры Ли \\$\mathbb{K}\partial x_1\oplus \ldots \oplus \mathbb{K}[x_1, \ldots , x_{n-1}]\partial x_n$. Тогда $\ker D = \mathbb{K}[x_1,\ldots, x_{n-1}]$ и для некоторого многочлена $f$, лежащего в алгебре $\mathbb{K}[x_1,\ldots, x_{n-1}]$, выполнено $D = f\partial x_n$.
    \end{thm}
    \begin{proof}
        
        По условию теоремы можем считать, что $D$ имеет вид $f_i\partial x_{i+1} +\ldots + f_{n-1}\partial x_n$, для некоторых $f_j\in\mathbb{K}[x_1, \ldots , x_j]$, причем $f_i\not= 0$. Также, по условию для каждого $a$ из $\ker D$, $aD = a f_i\partial x_{i+1} +\ldots + a f_{n-1}\partial x_n\in\mathbb{K}\partial x_1\oplus \ldots \oplus \mathbb{K}[x_1, \ldots , x_{n-1}]\partial x_n$, поэтому $(\ker D)f_i\subseteq \mathbb{K}[x_1, \ldots , x_i]$ и $\ker  D\subseteq \mathbb{K}[x_1, \ldots , x_i]$. С другой стороны $\mathbb{K}[x_1, \ldots , x_i]\subseteq\ker D$, а значит $\ker D = \mathbb{K}[x_1, \ldots , x_i]$ и $i=n-1$. Откуда $D = f_{n-1}\partial x_n$.
    \end{proof}
    
    \begin{cons}
        Гипотеза 2 не верна в случае $n=3$.
    \end{cons}
    \begin{proof}
        Рассмотрим локально нильпотентное дифференцирование $D = x\partial y + y\partial z$ алгебры многочленов $\mathbb{K}[x, y, z]$. В случае справедливости гипотезы 2 получаем, что для некоторых алгебраически независимых многочленов $x', y', z'$ выполнено $\mathbb{K}[x, y, z] = \mathbb{K}[x', y', z']$ и $(\ker D)D\subseteq\mathbb{K}\partial x'\oplus\mathbb{K}[x']\partial y\oplus\mathbb{K}[x', y']\partial z'$. Применяя теорему 2 к дифферренцированию $D$ получаем, что $\ker D = \mathbb{K}[x', y']$ и $\mathbb{K}[x, y, z]  = \mathbb{K}[x', y', z'] = (\ker D)[z'] =\mathbb{K}[x, 2xz - y^2, z']$. Что противоречит тому, что $\mathbb{K}[x, 2xz - y^2]$ является подалгеброй многочленов от двух переменных которую нельзя дополнить до алгебры $\mathbb{K}[x, y, z]$ присоединив ровно один многочлен.
    \end{proof}

    \section{Основные свойства локально нильпотентных множеств дифференцирований}
    
    \begin{thm}
        
        Пусть $S$~-- произвольное локально нильпотентное множество дифференцирований алгебры $B$, для кототорого известно, что $\ker S = \mathbb{K}$. Тогда найдутся $x_i$ такие, что $B = \mathbb{K}[x_1,\ldots , x_n]$ и $S\subseteq \mathbb{K}\partial x_1\oplus \ldots \oplus \mathbb{K}[x_1, \ldots , x_{n-1}]\partial x_n$.
        
    \end{thm}
    
    \begin{proof}
        
        Рассмотрим семейство $\mathcal{P}$ подалгебр $A$ алгебры $B$ обладающих следующими свойствами:
        
        \begin{enumerate}
            \item $A$ инвариантна для каждого дифференцирования из $S$;
            \item Существуют алгебраически независимые элементы $a_1, a_2,\ldots , a_k$ такие, что $A=\mathbb{K}[a_1,\ldots , a_k]$ и $S|_A\subseteq\mathbb{K}\partial a_1\oplus \ldots \oplus \mathbb{K}[a_1, \ldots , a_{k-1}]\partial a_k$.
        \end{enumerate}
        
        Введём частичный порядок на $\mathcal{P}$. Будем говорить, что для двух подалгебр $A_1, A_2\in\mathcal{P}$ выполнено $A_1\prec A_2$ тогда и только тогда, когда существуют элементы $a_{1}, a_{2},\ldots , a_{k_1}$ и $a_{k_1+1}, a_{k_1+2},\ldots ,$ $a_{k_2}$ такие, что $A_1 = \mathbb{K}[a_{1}, a_{2},\ldots , a_{k_1}], A_2 = A_1[a_{k_1+1}, a_{k_1+2},\ldots , a_{k_2}] = \mathbb{K}[a_{1},\ldots , a_{k_2}]$ и $S|_{A_2}\subseteq\mathbb{K}\partial a_{1}\oplus \ldots \oplus \mathbb{K}[a_{1}, \ldots , a_{k_2-1}]\partial a_{k_2}$.
        Яcно, что не существует бесконечно возрастающей цепочки подалгебр лежащих в $\mathcal{P}$, поэтому можно выбрать максимальную подалгебру $A$ из $\mathcal{P}$. Предположим, что $A\not= B$ и выберем произвольный элемент $b\in B\setminus A$ такой, что $Sb\subseteq A$ (такой элемент существует по лемме 1).
        Проверим, что $b$ алгебраически независим над $A$. Предположим противное: пусть для некоторых элементов $a_0, a_1, \ldots, a_d\in A$ выполнено $a_0 + a_1b+\ldots + a_db^d = 0$, причем $a_d\not= 0$. Можем считать, что $d>0$~-- минимально возможная степень аннулирующего многочлена. Заметим, что по построению алгебры $A$ и элемента $b$, элементы $D(a_j), D(b)$ лежат в $A$ для всех $a_j$ и всех дифференцирований $D$ лежащих в $S$. Имеем $0=D(a_0 + a_1b+\ldots + a_db^d) = (D(a_0)+a_1D(b))+\ldots + (D(a_{d-1})+da_dD(b))b^{d-1} + D(a_d)b^d$ для всех $D\in S$. Поэтому, в случае $a_d\notin\mathbb{K}$, подставляя в последнем равенстве вместо элемента $D\in S$ такой элемент $D'\in S$, что $a_d\notin\ker D'$, получим аннулирующий многочлен элемента $b$ степени $d$ со старшим коэффициентом равным $D'(a_d)\not=0$. Таким образом, из локальной нильпотентности множества дифференцирований $S$, после выполнения некоторого конечного числа таких операций получаем, что $b$ имеет аннулирующий многочлен степени $d$ со старшим коэффициентом лежащим в поле $\mathbb{K}$. Можем считать, что $a_0 + a_1b+\ldots + a_db^d$~-- искомый многочлен с $a_d\in\mathbb{K}\setminus\{0\}$. Заметим тогда, что для каждого дифференцирования $D$ из $S$ выполнено $0=D(a_0 + a_1b+\ldots + a_db^d) = (D(a_0)+a_1D(b))+\ldots + (D(a_{d-1})+da_dD(b))b^{d-1}$, а значит, из минимальности $d$, $D(a_{d-1})+da_dD(b) = D(a_{d-1}+da_db) = 0$ для всех $D\in S$. Откуда $a_{d-1}+da_db\in\mathbb{K}$ и $b\in A$, что противоречит определению элемента $b$.
        
        В этоге получаем, что $A[b]$ является большей алгеброй чем $A$, лежащей в $\mathcal{P}$. Противоречие. Значит алгебра $A$ совпадает с $B$ и для некоторых элементов $x_1, x_2,\ldots , x_n$, $B = \mathbb{K}[x_1,\ldots , x_n]$ и $S\subseteq \mathbb{K}\partial x_1\oplus \ldots \oplus \mathbb{K}[x_1, \ldots , x_{n-1}]\partial x_n$.
        \end{proof}
    
    \begin{cons}
        
        Пусть $S$~-- произвольное локально нильпотентное множество дифференцирований алгебры $B$. Тогда алгебра Ли $\mathcal{A}$ порожденная множеством $(\ker S)S$ является разрешимой алгеброй Ли, которая также является локально нильпотентным подмножеством дифференцирований алгебры $B$.
        
    \end{cons}
    
    \begin{proof}
        
        Переходя к алгебре $B'=(\ker S \setminus\{0\})^{-1}B$ над полем $(\ker S \setminus\{0\})^{-1}\mathbb{K}$ и локально нильпотентному подмножеству дифференцирований $S' = (\ker S\setminus\{0\})^{-1}S$, можем считать, что $\ker S = \mathbb{K}$. По теореме 3 найдутся $x_i$ такие, что $B = \mathbb{K}[x_1,\ldots , x_n]$ и $S\subseteq\mathbb{K}\partial x_1\oplus\ldots\oplus\mathbb{K}[x_1, \ldots , x_{n-1}]\partial x_n$. Утверждение следствия теперь следует из того, что $\text{d} (\mathbb{K}\partial x_1\oplus \ldots \oplus \mathbb{K}[x_1, \ldots , x_{n-1}]\partial x_n) = n$ и леммы 3. \end{proof}

\begin{thm}
    
    Множество дифференцирований $S$ алгебры $B$ является локально нильпотентным тогда и только тогда, когда каждое его конечное подмножество $S'\subseteq S$ является локально нильпотентным множеством дифференцирований алгебры $B$.
    
\end{thm}
\begin{proof}
    Докажем, что если каждое конечное подмножество некоторого множества дифференцирований $S$ является локально нильпотентным, то и само $S$ является локально нильпотентным множеством дифференцирований.

    Ведём индукцию по степени трансцендентности алгебры $B$. В случае $\text{tr.deg.}_{\mathbb{K}}(B) = 0$, имеем $\text{LND}(B)=\{0\}, S = \{0\}$ и утверждение леммы очевидно.
    Из следствия 2 получаем, что алгебра Ли $\mathcal{A}$ порожденная множеством $S$ также обладает тем свойством, что каждое её конечное подмножество образует локально нильпотентное множество дифференцирований алгебры $B$. Поэтому можем считать, что $S = \mathcal{A}$~-- алгебра Ли. Рассматривая алгебру Ли $\mathcal{A}' = (\ker\mathcal{A}\setminus\{0\})^{-1}\mathcal{A}$ над полем $\mathbb{K}' = (\ker\mathcal{A}\setminus\{0\})^{-1}\mathbb{K}$ и алгебру $B' = (\ker\mathcal{A}\setminus\{0\})^{-1}B$ можем считать, что $\ker\mathcal{A} = \mathbb{K}$. Применяя следствие 2, а также лемму 5 к множествам дифференцирований $\mathcal{A}$, $\mathcal{A}^{(1)}$ получаем, что существует достаточно большая конечномерная разрешимая подалгебра Ли $\mathcal{B}\subset\mathcal{A}$ являющаяся локально нильпотентным множеством дифференцирований алгебры $B$ такая, что $\ker\mathcal{B} = \ker\mathcal{A} = \mathbb{K}$ и $\ker\mathcal{B}^{(1)} = \ker\mathcal{A}^{(1)}$. Поэтому из теоремы 3 получаем, что для некоторых элементов $x_i$, $B = \mathbb{K}[x_1,\ldots, x_n]$ и $\mathcal{B}\subseteq \mathbb{K}\partial x_1\oplus \ldots \oplus \mathbb{K}[x_1, \ldots , x_{n-1}]\partial x_n$. Таким образом, $\mathbb{K}[x_1] \subseteq \ker(\mathbb{K}\partial x_1\oplus \ldots \oplus \mathbb{K}[x_1, \ldots , x_{n-1}]\partial x_n)^{(1)}\subseteq \ker\mathcal{B}^{(1)} = \ker\mathcal{A}^{(1)}$. Пусть $\pi:\mathcal{A}\to\mathcal{A}|_{\ker\mathcal{A}^{(1)}}$~-- гомоморфизм ограничения алгебры Ли $\mathcal{A}$ на инвариантную подалгебру $\ker\mathcal{A}^{(1)}$. Понятно, что каждое дифференцирование, лежащее в $\ker\pi$, содержит $\ker\mathcal{A}^{(1)}$, а следовательно и $\mathbb{K}[x_1]$, в своём ядре. Рассмотрим алгебру $B' = (\mathbb{K}[x_1]\setminus\{0\})^{-1}B$ степени трансцендентности не более чем $\text{tr.deg}_{\mathbb{K}}(B)-1$, и алгебру Ли $\mathcal{P} = (\mathbb{K}[x_1]\setminus\{0\})^{-1}\ker\pi$. Алгебра Ли $\mathcal{P}$ является множеством дифференцирований алгебры $B'$ меньшей степени трансцендентности чем $B$, откуда по предположению индукции $\mathcal{P}$~-- локально нильпотентное множество дифференцирований алгебры $B'$, а значит, $\ker\pi$~-- локально нильпотентное множество дифференцирований алгебры $B$. Заметим, что $\mathcal{A}^{(1)}\subseteq\ker\pi$, поэтому $\mathcal{A}|_{\ker\mathcal{A}^{(1)}}\cong\mathcal{A}/\ker\pi$ является коммутативной алгеброй Ли лежащей в $\text{LND}(\ker\mathcal{A}^{(1)})$ (мы воспользовались условием теоремы для одноэлементных подмножеств). По лемме 6 получаем, что $\mathcal{A}/\ker\pi$~-- конечномерная алгебра Ли, откуда $\ker\pi$~-- локально нильпотентное множество дифференцирований алгебры $B$ имеющее конечную коразмерность в $\mathcal{A}$. Применяя лемму 4 к множеству $\ker\pi$ и произвольному конечному базису факторпространства $\mathcal{A}/\ker\pi$ получаем, что $\mathcal{A}$~-- локально нильпотентное множество дифференцирований алгебры $B$.
\end{proof}
\begin{cons}
    Пусть $S_1\subset S_2\subset\ldots$~-- возрастающая цепочка вложенных друг в друга локально нильпотентных множеств дифференцирований алгебры $B$. Тогда $\cup_{i\geq 1} S_i$ также является локально нильптентным множеством дифференцирований алгебры $B$.
\end{cons}
    \begin{thm}
            Пусть $S$~-- произвольное локально нильпотентное множество дифференцирований алгебры $B$. Тогда множество $ad(S)$ является локально нильпотентным множеством линейных операторов на векторном пространсве $\text{Der}(B)$.
    \end{thm}
    
    \begin{proof}
        
        Рассмотрим алгебру $B' = (\ker S\setminus\{0\})^{-1}B$ над полем $\mathbb{K}' = (\ker S\setminus\{0\})^{-1}\mathbb{K}$ и локально нильпотентное множество дифференцирований $S' = (\ker S\setminus\{0\})^{-1}S$ алгебры $B'$. Из теоремы 3 следует, что существуют $x_i$ такие, что $B' = \mathbb{K}'[x_1,\ldots, x_n]$, $S\subseteq\mathbb{K}'\partial x_1\oplus\mathbb{K}'[x_1]\partial x_2\oplus\ldots\oplus\mathbb{K}'[x_1,\ldots, x_{n-1}]\partial x_n$. Умножая элементы $x_i\in B'$ на подходящие элементы из $\ker S\subseteq\mathbb{K}'$, можем считать, что $x_i\in B$. Пусть $\{b_1, b_2,\ldots, b_k\}$~-- базис трансцендентности алгебры $\ker S$. Для каждого дифференцирования $D\in\text{Der}(B)$ рассмотрим его "координаты" $D_1 = D(b_1), D_2 = D(b_2),\ldots, D_k = D(b_k), D_{k+1} =  D(x_1), D_{k+2} = D(x_2),\ldots, D_{k+n} = D(x_n)$. Высотой дифференцирования $D$ будем называть максимальное число $t>0$ такое, что $D_1 = D_2 = \ldots = D_{t-1} = 0$. В случае если дифференцирование $D$ имеет высоту равную $n+k+1$, имеем $\ker D\supseteq\mathbb{K}[b_1,\ldots, b_k, x_1,\ldots, x_n]$~-- алгебра степени трансцендентности равной $\text{tr.deg.}_{\mathbb{K}}(B)$ являющаяся алгебраически замкнутой в $B$ (мы воспользовались \cite[Prop. 1.8(d); стр. 17]{Fr}), а следовательно, $\ker D = B$, $D=0$. Введём частичный порядок на множестве $\text{Der}(B)$. Для произвольных дифференцирований $D, E\in\text{Der}(B)$ будем говорить, что $D\prec E$ если высота $h$ дифференцирования $D$ не меньше высоты дифференцироания $E$ и порядок $h$-ой координаты дифференцирования $D$ как многочлена в $\mathbb{K}'[x_1, \ldots, x_n]$ (считаем, что на многочленах уже введён лексикографический порядок с $x_1\prec x_2\prec\ldots\prec x_n$) строго меньше порядка соответсвующей координаты дифференцирования $E$. Несложно проверяется, что не существует бесконечной строго убывающей последовательности дифференцирований относительно введённого нами порядка, а также, что для произвольных двух дифференцирований $0\not= D\in\text{Der}(B)$, $\partial\in S$, $(\text{ad}\partial)(D)\prec D$. Откуда следует, что $\text{ad}S$~-- локально нильпотентное множество линейных операторов на векторном пространсве $\text{Der}(B)$.
    \end{proof}
    
    \section{Доказательство теоремы 1}
    
    Мы докажем несколько более сильное утверждение, имеющее место для всех коммутативных алгебр $B$ без делителей нуля, конечной степени трансцендентности.
    
    \begin{thm}
        
        Пусть дана алгебра Ли $\mathcal{A}$, лежащая в $\text{LND}(B)$ и известно, что $\ker\mathcal{A} = \mathbb{K}$. Тогда найдутся $x_i$ такие, что $B = \mathbb{K}[x_1,\ldots , x_n]$ и $\mathcal{A}\subseteq\mathbb{K}\partial x_1\oplus \ldots \oplus \mathbb{K}[x_1, \ldots , x_{n-1}]\partial x_n$.
        
    \end{thm}    
    
    Для доказательства теоремы 6 остаётся показать, что всякая алгебра Ли, лежащая в $\text{LND}(B)$, образует локально нильпотентное множество дифференцирований алгебры $B$, после чего применить теорему 3.
    
    \begin{st}
        
        Всякая алгебра Ли $\mathcal{A}$, лежащая в $\text{LND}(B)$, образует локально нильпотентное множество дифференцирований алгебры $B$.
        
    \end{st}
    
    \begin{proof}
        По следствию 3 каждая возрастающая цепочка вложенных локально нильпотентных множеств дифференцирований алгебры $B$ лежащих в $\mathcal{A}$ имеет верхнюю грань, поэтому по лемме Цорна существует максимальное по вложению локально нильпотентное множество дифференцирований $S\subseteq\mathcal{A}$. Предположим, что $S\not=\mathcal{A}$. По теореме 5, $\text{ad}S$~-- локально нильпотентное множество линейных операторов на $\text{Der}(B)$. Применим лемму 1 к локально нильпотентному множеству линейных операторов $(\text{ad}S)|_{\mathcal{A}}$ на векторном пространстве $\mathcal{A}$ и собственному подпространству $S\subsetneqq\mathcal{A}$. Получим, что для некоторого элемента $D\in\mathcal{A}\setminus S$, $(\text{ad}S)(D) = [S, D]\subseteq S$, откуда по лемме 4, $S\cup\{D\}$~-- локально нильпотентное множество дифференцирований большее чем $S$. Противоречие.
    \end{proof}

    \section{Доказательство гипотезы 1}
    
    \begin{thm}
        
        Алгебра Ли $\frak{T} = \mathbb{K}\partial x_1\oplus \ldots\oplus \mathbb{K}[x_1, \ldots, x_{n-1}]\partial x_n$ является максимальной по вложению алгеброй Ли, лежащей в $\text{LND}(\mathbb{K}[x_1, \ldots, x_n])$.
        
    \end{thm}
    
    \begin{proof}
        
        Рассмотрим максимальную по вложению алгебру Ли $\mathcal{A}$, вложенную в \\$\text{LND}(\mathbb{K}[x_1, \ldots, x_n])$ такую, что $\ker\mathcal{A} = \mathbb{K}$ (в качестве примера такой алгебры Ли можно рассмотреть произвольную максимальную по вложению алгебру Ли содержащую $\frak{T}$ и лежащую в $\text{LND}(\mathbb{K}[x_1, \ldots, x_n])$). По теореме 1 получаем, что алгебра Ли $\mathcal{A}$ сопряжена некоторой подалгебре $\mathcal{A}'$ алгебры Ли $\frak{T}$. Из максимальности $\mathcal{A}$ имеем $\frak{T} = \mathcal{A}'$~-- максимальная по вложению алгебра Ли, лежащая в $\text{LND}(\mathbb{K}[x_1, \ldots, x_n])$. \end{proof}


\end{document}